\documentclass[oneside,english,british]{amsart}
\usepackage[T1]{fontenc}
\usepackage[utf8]{inputenc}
\usepackage{dsfont}
\usepackage{amstext}
\usepackage{amsthm}
\usepackage{amssymb}
\usepackage{hyperref}

\makeatletter
\numberwithin{equation}{section}

\usepackage{alex}
\usepackage{dsfont} % to get \mathds{1} - "mathbb style 1"
\DeclareMathOperator{\St}{St}

\makeatletter
\newcommand*{\transpose}{%
  {\mathpalette\@transpose{}}%
}
\newcommand*{\@transpose}[2]{%
  \raisebox{\depth}{$\m@th#1\intercal$}%
}
\makeatother

\makeatother

\theoremstyle{remark}
\newtheorem*{notation*}{\protect\notationname}
\theoremstyle{plain}
\newtheorem{thm}{\protect\theoremname}[section]
\newtheorem{lem}[thm]{\protect\lemmaname}
\theoremstyle{remark}
\newtheorem{rem}[thm]{\protect\remarkname}
\theoremstyle{plain}
\newtheorem{prop}[thm]{\protect\propositionname}
\theoremstyle{definition}
\newtheorem{defn}[thm]{\protect\definitionname}
\newtheorem{example}[thm]{\protect\examplename}
\theoremstyle{remark}
\newtheorem*{rem*}{\protect\remarkname}
\theoremstyle{plain}
\newtheorem{cor}[thm]{\protect\corollaryname}
\usepackage{babel}
\addto\captionsbritish{\renewcommand{\corollaryname}{Corollary}}
\addto\captionsbritish{\renewcommand{\definitionname}{Definition}}
\addto\captionsbritish{\renewcommand{\examplename}{Example}}
\addto\captionsbritish{\renewcommand{\lemmaname}{Lemma}}
\addto\captionsbritish{\renewcommand{\notationname}{Notation}}
\addto\captionsbritish{\renewcommand{\propositionname}{Proposition}}
\addto\captionsbritish{\renewcommand{\remarkname}{Remark}}
\addto\captionsbritish{\renewcommand{\theoremname}{Theorem}}
\addto\captionsenglish{\renewcommand{\corollaryname}{Corollary}}
\addto\captionsenglish{\renewcommand{\definitionname}{Definition}}
\addto\captionsenglish{\renewcommand{\examplename}{Example}}
\addto\captionsenglish{\renewcommand{\lemmaname}{Lemma}}
\addto\captionsenglish{\renewcommand{\notationname}{Notation}}
\addto\captionsenglish{\renewcommand{\propositionname}{Proposition}}
\addto\captionsenglish{\renewcommand{\remarkname}{Remark}}
\addto\captionsenglish{\renewcommand{\theoremname}{Theorem}}
\providecommand{\corollaryname}{Corollary}
\providecommand{\definitionname}{Definition}
\providecommand{\examplename}{Example}
\providecommand{\lemmaname}{Lemma}
\providecommand{\notationname}{Notation}
\providecommand{\propositionname}{Proposition}
\providecommand{\remarkname}{Remark}
\providecommand{\theoremname}{Theorem}

\begin{document}
\title{A tensor square theorem for characters of $\GL_{n}(q)$\\
}
\author{\selectlanguage{english}%
Nariel Monteiro and Alexander Stasinski}
\selectlanguage{british}%
\begin{abstract}
We construct an irreducible character of $\GL_{n}(q)$ whose square
contains every irreducible character that is trivial on the centre
of $\GL_{n}(q)$. This character is related to, but in general not
equal to, the Steinberg character. A key ingredient, which we also
prove, is that every nonlinear irreducible character of $\GL_{n}(q)$
that is trivial on the centre contains the trivial character when
restricted to the diagonal subgroup $T$.
\end{abstract}

\address{\selectlanguage{english}%
Department of Mathematics, University of California, Santa Cruz, CA
95064, USA}
\address{\selectlanguage{english}%
Department of Mathematical Sciences, Durham University, Durham, DH1
3LE, UK}
\email{\selectlanguage{english}%
namontei@ucsc.edu\\
alexander.stasinski@durham.ac.uk}
\maketitle

\section{Introduction}

Given a finite group $G$, the problem of the existence of an irreducible
character of $G$ whose square contains every irreducible character
of $G$ was first studied by Heide, Saxl, Tiep and Zalesski \cite{Heide-Saxl-Tiep-Zalesski},
who showed that if $G$ is a finite simple group of Lie type, other
than $\mathrm{PSU}_{n}(q)$ with $n\geq3$ coprime to $2(q+1)$, then
the square of the Steinberg character contains every irreducible character
of $G$. The problem was initially partly motivated by the analogy
with Thompson's conjecture that every finite simple group possesses
a conjugacy class whose square covers the group. Moreover, it has
been shown by direct computation that the sporadic simple groups each
have an irreducible character whose square contains every irreducible
character. As far as simple groups are concerned, this leaves only
the alternating groups, for which the question is open and is related
to the Saxl conjecture, which has recently received a great deal of
attention (see, e.g., \cite{Pak_et_al-Saxl_conj}).

In the present paper, we take the first step towards extending the
result in \cite{Heide-Saxl-Tiep-Zalesski} to not necessarily simple,
finite groups of Lie type, where the first obvious candidate is $\GL_{n}(q)$.
In this case, it is easily seen that the square of the Steinberg character,
or any other irreducible character, does not contain every irreducible
character in general. Indeed, the square of an irreducible character
can only contain irreducible constituents with a fixed central character,
so whenever the centre of the group is nontrivial, no such square
can contain all the irreducibles.  Since the Steinberg character
has trivial central character, one may therefore in general ask whether
there exists an irreducible character of $G=\GL_{n}(q)$ whose square
contains every irreducible character of $G$ that is trivial on the
centre $Z$ of $G$. In the present paper, we show that this is true,
although the square of the Steinberg character is not sufficient in
general. More precisely, we define the irreducible Harish--Chandra/parabolically
induced character
\[
\sigma=R^{G}_{L}(\St_{m},\epsilon\St_{m},\dots,\epsilon^{d-1}\St_{m})=\Ind^{G}_{P}\Inf^{P}_{L}(\St_{m},\epsilon\St_{m},\dots,\epsilon^{d-1}\St_{m}),
\]
where $d=\gcd(n,q-1)$, $m=n/d$, $L=\GL_{m}(q)^{d}$ is the block
diagonal subgroup with $d$ blocks of size $m$, $P$ the corresponding
block-upper subgroup, $\St_{m}$ denotes the Steinberg character of
$\GL_{m}(q)$ and $\epsilon\in\Irr(\F^{\times}_{q})$ is a character
of order $d$, the powers of which are used to twist $\St_{m}$ on
each block of $L$. Here $\Inf^{P}_{L}$ denotes inflation from $L$
to $P$. One of our main results
is that $\sigma^{2}$ contains every irreducible character of $G$
that is trivial on $Z$ (see Theorem~\ref{thm: Main thm 2}). When
$d=1$, $\PGL_{n}(q)\cong\text{PSL}_{n}(q)$, and $\sigma=\St_{n}$,
so in this case our result coincides with \cite[Theorem~1.2]{Heide-Saxl-Tiep-Zalesski}.

We note that even the generalised `tensor square' question, where
one considers only characters trivial on the centre, may fail to have
a positive answer for general finite groups of Lie type. Namely, we
show (see Example~\ref{exa:PGL2}) that $\PGL_{2}(q)$ with $q$
odd and $q\not\equiv1\pmod 4$, does not have any irreducible character
whose square contains every irreducible character. Nevertheless, it
follows from our result that $\GL_{2}(q)$ has an irreducible character
whose square contains every irreducible character of $\PGL_{2}(q)$
(identifying characters of $\PGL_{2}(q)$ with characters of $\GL_{2}(q)$
that are trivial on the centre). More generally, one may ask when
a non-simple adjoint finite group of Lie type has a central extension
playing the role of $\GL_{n}(q)$ and a character playing the role
of $\sigma$.

A key ingredient of the proof of Theorem~\ref{thm: Main thm 2} is
our other main result, Theorem~\ref{thm: Restriction to T}, that
every nonlinear irreducible character of $G$ that is trivial on $Z$
contains the trivial character when restricted to the diagonal subgroup
$T$. In particular, it follows directly from this and a result of
Deligne and Lusztig that the Steinberg square contains all the nonlinear
characters of $G$ that are trivial on $Z$. However, it is easy to
see that already in the case of $\GL_{2}(q)$ with $q$ odd, the Steinberg
square misses a linear character that is trivial on the centre, hence
the need for another character, such as our $\sigma$ above. 

We believe
that Theorem~\ref{thm: Restriction to T} may have nontrivial generalisations
in several directions. For example, in \cite{Zalesski2016}, Zalesski
showed an analogous result on restriction to the `Singer torus' (i.e.,
the nonsplit/elliptic torus) of $G$, by a very different method of
proof (and motivation different from ours). In Zalesski's result,
one has to exclude not only linear characters but also those of the
next higher degree. One may therefore speculate that something similar
holds on restriction to any maximal torus in $G$, and possibly also
for other reductive groups over $\F_{q}$ (or over finite rings).

While we focus on characters with trivial central character, one may
ask whether, for any square character $\zeta\in\Irr(Z)$, there exists
a $\rho\in\Irr(G)$ whose square contains all the irreducible characters
of $G$ with central character $\zeta$. In the end of the paper,
we show that Theorem~\ref{thm: Main thm 2} immediately implies a
positive answer when a square root of $\zeta$ extends to $G$, but
that the answer can otherwise be negative.

The study of the irreducible constituents of tensor products of representations
of finite groups of Lie type and related groups has recently attracted
considerable attention. We mention here only a small selection of
the most recent work. Gupta--Hassain \cite{Gupta2025} explicitly
decompose the tensor product of any two irreducible representations
of $\GL_{2}(q)$, following previous results of \cite{Aburto-Hageman-PantojaGL2}.
Letellier and Nam \cite{Letellier2025} prove an analogue of the Saxl
conjecture for unipotent representations of $\GL_{n}(q)$. Nam \cite{Nam2026}
gives formulas for the multiplicities of tensor products of unipotent
almost characters and Deligne--Lusztig characters of a split untwisted
reductive group over a finite field; in particular, giving criteria
for when the Steinberg square contains or does not contain, certain
Deligne--Lusztig characters.
\subsection*{Notation}
We will denote the trivial character of a given finite group by $\mathds{1}$,
but occasionally specify the group in the notation for clarity, writing
for example $\mathds{1}_{T}$. For two characters $\rho$ and $\rho'$
of a finite group, we will also write $\rho\subseteq\rho'$ if $\rho$
is the character of a subrepresentation of a representation with character
$\rho'$ (note that this is stronger than just requiring that every
irreducible constituent of $\rho$ is contained in $\rho'$). 

If $\chi\in\Irr(\F^{\times}_{q})$ and $\rho\in\Irr(\GL_{n}(q))$,
we will write $\chi\rho$ for the twisted character $(\chi\circ\det)\rho$
of $\GL_{n}(q)$.

If $L=\GL_{m_{1}}(q)\times\dots\times\GL_{m_{k}}(q)$ is a block diagonal
subgroup of $\GL_{n}(q)$ and $\rho_{i}\in\Irr(\GL_{m_{i}}(q))$,
we will usually write the character $\rho_{1}\boxtimes\dots\boxtimes\rho_{k}$
of $L$ as $(\rho_{1},\dots,\rho_{k})$.

\subsection*{Acknowledgements}
The first-named author was supported by the National Science Foundation MPS-Ascend Postdoctoral Research Fellowship under Grant No. 2213166, and gratefully acknowledges additional financial support from the Department of Mathematics at the University of California, Santa Cruz.

\section{Restriction to the diagonal subgroup}

In this section we prove one of the key results needed for the other
main theorem of this paper. To state it, let $G=\GL_{n}(q)$, let
$T$ be the diagonal subgroup of $G$ and let $Z$ be the centre of
$G$. An obvious necessary condition for a $\rho\in\Irr(G)$ to contain
$\mathds{1}_{T}$ when restricted to $T$ is that $\rho$ contains
$\mathds{1}_{Z}$ when restricted to $Z$. Another straightforward
necessary condition is that $\rho$ is not linear, unless it is trivial
or $G=\GL_{2}(2)$. We prove that in general these conditions are
also sufficient.

For any $n\geq2$, let $P$ be $(n-1,1)$-block-upper-triangular subgroup
of $G$ and $V$ its block-upper unipotent subgroup, that is,
\[
P=\begin{pmatrix}\GL_{n-1}(q) & *\\
\bar{0}^{\transpose} & *
\end{pmatrix},\qquad V=\begin{pmatrix}I & *\\
\bar{0}^{\transpose} & 1
\end{pmatrix},
\]
where $I$ is the identity matrix of size $n-1$. Note that $V$ is
abelian and normal in $P$. We also have the semidirect product $P=LV$,
where $L$ is the $(n-1,1)$-block-diagonal subgroup. 
\begin{lem}
\label{lem: chars of V and action on them}Let $n\geq2$, with notation
as above.
\end{lem}

\begin{enumerate}
\item Let $\psi:\F^{+}_{q}\rightarrow\C^{\times}$ be a nontrivial additive
homomorphism. There exists an isomorphism
\[
\bfa\in\F^{n-1}_{q}\longiso\Irr(V),\qquad\bfa\longmapsto[\psi_{\bfa}:v\longmapsto\psi(\bfa\cdot v)],
\]
where we identify a column vector $v\in\F^{n-1}_{q}$ with its corresponding
element $\begin{pmatrix}I & v\\
\bar{0}^{\transpose} & 1
\end{pmatrix}\in V$ . Moreover, for $g=\begin{pmatrix}M & \bar{0}\\
\bar{0}^{\transpose} & c
\end{pmatrix}\in L$, where $M\in\GL_{n-1}(q)$ and $c\in\F^{\times}_{q}$,
\[
\leftexp{g}{\psi_{\bfa}}=\psi_{c(M^{-1})^{\transpose}\bfa}
\]
\item There are exactly two $P$-orbits in $\Irr(V)$. Moreover, the nontrivial
orbit contains an element $\psi_{*}$ such that $\Stab_{T}(\psi_{*})=Z$.
\end{enumerate}
\begin{proof}
(1): The first assertion follows from the fact that the dot product
defines a nondegenerate bilinear form. For the second assertion, for
any $v\in V$,
\begin{align*}
\leftexp{g}{\psi_{\bfa}}(v) & =\psi_{\bfa}(g^{-1}vg)=\psi_{\bfa}\Big(\begin{pmatrix}M^{-1} & \bar{0}\\
\bar{0}^{\transpose} & c
\end{pmatrix}\begin{pmatrix}I & v\\
\bar{0}^{\transpose} & 1
\end{pmatrix}\begin{pmatrix}M & \bar{0}\\
\bar{0}^{\transpose} & c
\end{pmatrix}\Big)=\psi_{\bfa}\Big(\begin{pmatrix}I & M^{-1}vc\\
\bar{0}^{\transpose} & 1
\end{pmatrix}\Big)\\
 & =\psi(\bfa\cdot M^{-1}vc)=\psi(c\bfa^{\transpose}M^{-1}v)=\psi((c(M^{-1})^{\transpose}\bfa)\cdot v)\\
 & =\psi_{c(M^{-1})^{\transpose}\bfa}(v).
\end{align*}
(2): As $V$ acts trivially on $\Irr(V)$, we only need to consider
the $L$-action, which is given in the previous part. The trivial
character $\mathds{1}_{V}=\psi_{\bar{0}}$ is its own orbit and since
the action of $\GL_{n-1}(q)$ on $\F^{n-1}_{q}\setminus\{\bar{0}\}$
is transitive, there is exactly one nontrivial orbit. 

Finally, let $\bfa=(1,\dots,1)^{\transpose}$ and set $\psi_{*}:=\psi_{\bfa}$.
By the previous part, $g=\begin{pmatrix}M & \bar{0}\\
\bar{0}^{\transpose} & c
\end{pmatrix}\in L$, lies in $\Stab_{L}(\psi_{*})$ if and only if $c(M^{-1})^{\transpose}\bfa=\bfa$.
Assuming that $g\in T$, we have that $M=\diag(m_{1},\dots,m_{n-1})$
is diagonal, hence $c(M^{-1})^{\transpose}\bfa=\bfa$ if and only
if $cm^{-1}_{i}=1$, for all $i=1,\dots,n-1$. Thus $g\in\Stab_{T}(\psi_{*})$
if and only if $g\in Z$.
\end{proof}

\begin{lem}
\label{lem: constits of G-G type reps satisfy what we want}Let $H$
be a subgroup of $G$ that is normalised by $T$ and such that $H\cap T=\{1\}$.
Suppose that there exists a character $\phi\in\Irr(H)$ such that
$\Stab_{T}(\phi)=Z$. Then every nonlinear irreducible constituent
of $\Ind^{G}_{H}\phi$ that is trivial on $Z$ contains $\mathds{1}_{T}$
on restriction to $T$. 
\end{lem}

\begin{proof}
Let $\rho\in\Irr(G)$ be a constituent of $\Ind^{G}_{H}\phi$ that
is trivial on $Z$. Then $\rho$ is contained in $\Ind^{G}_{ZH}\tilde{\phi}$,
where $\tilde{\phi}$ is the extension of $\phi$ to $ZH$ that is
trivial on $Z$. Thus $0\neq\langle\rho,\Ind^{G}_{TH}\Ind^{TH}_{ZH}\tilde{\phi}\rangle=\langle\rho|_{TH},\Ind^{TH}_{ZH}\tilde{\phi}\rangle$.
Since $\Stab_{T}(\phi)=Z$, we have $\Stab_{TH}(\tilde{\phi})=ZH$,
so $\Ind^{TH}_{ZH}\tilde{\phi}$ is irreducible and is hence contained
in $\rho|_{TH}$. On the other hand, by Mackey's induction-restriction
formula,
\[
\Res^{TH}_{T}\Ind^{TH}_{ZH}\tilde{\phi}\supseteq\Ind^{T}_{T\cap(ZH)}(\tilde{\phi}|_{T\cap(ZH)})=\Ind^{T}_{Z}(\tilde{\phi}|_{Z})\supseteq\Ind^{T}_{Z}\mathds{1}.
\]
Thus $\rho|_{T}\supseteq\Ind^{T}_{Z}\mathds{1}\supseteq\mathds{1}_{T}$.
\end{proof}

\begin{thm}
\label{thm: Restriction to T}Let $\rho\in\Irr(G)$ be a nonlinear
character that is trivial on $Z$. Then $\rho|_{T}$ contains $\mathds{1}_{T}$.
\end{thm}

\begin{proof}
Let $H=V$ and $\phi=\psi_{*}$ as in Lemma~\ref{lem: chars of V and action on them}.
Then Lemma~\ref{lem: constits of G-G type reps satisfy what we want}
implies that every nonlinear irreducible constituent of $\Ind^{G}_{V}\psi_{*}$
that is trivial on $Z$ contains $\mathds{1}_{T}$ on restriction
to $T$. If $\rho|_{V}$ contains a nontrivial character $\psi_{\bfa}$
of $V$, then as $\rho$ is $P$-stable, $\rho|_{V}$ contains all
the characters in the $P$-orbit of $\psi_{\bfa}$, in particular
$\psi_{*}$. Therefore, by Lemma~\ref{lem: chars of V and action on them},
any $\rho\in\Irr(G)$ either contains $\psi_{*}$ or $\mathds{1}_{V}$
on restriction to $V$.

Assume now that $\rho|_{V}$ contains only $\mathds{1}_{V}$. Then
$V\subseteq\Ker\rho$, so $\Ker\rho$ is a non-central normal subgroup
of $G$. Note that $V$ contains the elementary transvection $I+aE_{1n}$,
$a\in\F_{q}$ and the permutation matrices in $G$ act transitively,
by conjugation, on the elements $I+E_{ij}$, $i\neq j$. Since $\SL_{n}(q)$
is generated by the elementary transvections $I+aE_{ij}$, the normal
closure of $V$, hence $\Ker\rho$, contains $\SL_{n}(q)$. Thus $\rho$
is a linear character of $G$.
\end{proof}

\begin{rem}\label{rem:Ind}
Inspecting the proof of Lemma~\ref{lem: constits of G-G type reps satisfy what we want}
and Theorem~\ref{thm: Restriction to T}, one can see that what is
actually proved is that the conclusion is that $\rho|_{T}\supseteq\Ind^{T}_{Z}\mathds{1}$.
It is also easy to see that a straightforward modification of the
proofs yields the more general statement that every nonlinear $\rho\in\Irr(G)$
with central character $\zeta$ contains $\Ind^{T}_{Z}\zeta$. We
shall however not need this.
\end{rem}

We will use the following lemma in the proof of Theorem~\ref{thm: Main thm 2}. Together with Theorem~\ref{thm: Restriction to T}, it immediately implies that every nonlinear irreducible character of $G$ that is trivial on $Z$ is contained the square of the Steinberg character.
\begin{lem}
\label{lem: St square contains all non-lin triv on Z}Every $\rho\in\Irr(G)$ that contains
$\mathds{1}_{T}$ on restriction to $T$ is contained in $\St^{2}_n$.
\end{lem}

\begin{proof}
By a formula of Deligne and Lusztig \cite[7.15.2]{delignelusztig},
$\St^{2}_n=\sum_{(\mathbf{T})}a_{i}\Ind^{G}_{\mathbf{T}(\F_{q})}\mathds{1}$,
where $a_{i}$ are positive rational numbers and the sum is taken
over a set of representatives of $G$-conjugacy classes of rational
maximal tori. Thus by Frobenius reciprocity, $\langle\rho,\St^{2}_n\rangle>\langle\rho,\Ind^{G}_{T}\mathds{1}_{T}\rangle=\langle\Res^{G}_{T}\rho,\mathds{1}_{T}\rangle>0$.
\end{proof}

\subsection*{A question on maximal tori}

Our result for the diagonal torus, together with Zalesski's result for a Singer torus \cite{Zalesski2016}, suggests the following question.

\medskip
\noindent\textbf{Question.}
Let $T$ be a maximal torus of $G$ (not necessarily the diagonal torus). If $\chi\in\operatorname{Irr}(G)$ is trivial on $Z$, does
$$
\chi(1)\geq |T/Z|
\quad\Longrightarrow\quad
\langle \chi|_T,\mathds{1}_T\rangle>0
$$ always hold?
\medskip

When $T$ is the diagonal torus, one can show (cf.~Remark~\ref{rem:Ind}) that $\chi\in \Irr(G)$ is nonlinear if and only if $\chi(1)\geq |T/Z|$ (even though this is not the best possible bound). Therefore, in this case, a positive answer to the above question follows from Theorem~\ref{thm: Restriction to T}. 

For a Singer torus, the implication follows from \cite[Theorem~1.1]{Zalesski2016}. Computations in GAP verify the implication for every maximal torus type in $\mathrm{GL}_3(3)$, $\mathrm{GL}_3(4)$, $\mathrm{GL}_3(5)$, and $\mathrm{GL}_4(3)$.

Some restriction on the degree is necessary. Indeed, in $\mathrm{PGL}_4(3)$ there is an irreducible character $\psi$ of  degree $39$ which does not contain the trivial character when restricted to the Singer torus $\overline{T}$ of order $40$. In particular, the proposed degree threshold cannot be lowered by one in general.

\section{The character $\sigma$ and its square}

We continue to write $G$ for $\GL_{n}(q)$. In this section we construct
a certain irreducible character $\sigma$ of $G$ and state our second
main result, namely that $\sigma^{2}$ contains every irreducible
character of $G$ that is trivial on the centre. The proof appears
in the following section. From now on, let 
\[
d:=\gcd(n,q-1).
\]

\begin{lem}
\label{lem:The-maximal-order-of-lin-char}There exists an $\epsilon\in\Irr(\F^{\times}_{q})$
such that $\epsilon\circ\det\in\Irr(G)$ is trivial on $Z$ and $\ord(\epsilon\circ\det)=d$.
Moreover, when $G\neq\GL_{2}(2)$, then the maximal order of any linear
character of $G$ that is trivial on $Z$ is $d$. 
\end{lem}

\begin{proof}
For $G=\GL_{2}(2)$, $d=1$ and $\epsilon=\mathds{1}$. We now assume
that $G\neq\GL_{2}(2)$ and prove both assertions of the lemma. Since
$G\neq\GL_{2}(2)$, we have $[G,G]=\SL_{n}(q)$, so any linear character
of $G$ is of the form $\epsilon\circ\det\in\Irr(G)$ for $\epsilon\in\Irr(\F^{\times}_{q})$.
Since $\epsilon^{q-1}=\mathds{1}$, we have $\ord(\epsilon\circ\det)\mid q-1$
and the condition that $\epsilon\circ\det$ is trivial on $Z$ is
equivalent to $(\epsilon\circ\det)^{n}=\mathds{1}$, so also $\ord(\epsilon\circ\det)\mid n$.
Thus $\ord(\epsilon\circ\det)\mid d$, so the maximal possible order
is $d$. On the other hand, since $\Irr(\F^{\times}_{q})$ is cyclic,
it has an element $\epsilon$ of order $d$ and since $\det:G\rightarrow\F^{\times}_{q}$
is surjective, this implies that $\epsilon\circ\det$ has order $d$.
\end{proof}

\begin{defn}
\label{def: Char sigma}Fix $n$ and $q$. Let $P$ be the block-upper
subgroup of $G=\GL_{n}(q)$ with $d$ blocks of size $m:=n/d$ and
let $\epsilon$ be an element of $\Irr(\F^{\times}_{q})$ such that
$\epsilon\circ\det$ is of order $d$ (which exists by Lemma~\ref{lem:The-maximal-order-of-lin-char}).
Let $\St_{m}$ denote the Steinberg character of $\GL_{m}(q)$ and
define the parabolically induced character
\[
\sigma=R^{G}_{L}(\St_{m},\epsilon\St_{m},\dots,\epsilon^{d-1}\St_{m}).
\]
\end{defn}

As usual, $\St_{1}$ of $\GL_{1}(q)$ is the trivial character. Note
that when $d=1$, $\sigma=\St_{n}$. We also observe that $\sigma$
is irreducible because the $d$ components of the inducing character
are distinct, as the following lemma shows.
\begin{lem}
Assume that $d>1$. Then $\epsilon^{i}\St_{m}=\epsilon^{j}\St_{m}$,
for some $0\leq i,j\leq d-1$ only if $i=j$. Thus $\sigma$ is irreducible.
\end{lem}

\begin{proof}
Suppose that $\epsilon^{i}\St_{m}=\epsilon^{j}\St_{m}$ with $i\geq j$.
Then 
\[
\langle\epsilon^{i}\St_{m},\epsilon^{j}\St_{m}\rangle=\langle\epsilon^{i-j}\St_{m},\St_{m}\rangle,
\]
so it is enough to prove that $\epsilon^{i}\St_{m}\neq\St_{m}$, for
every $1\leq i\leq d-1$, that is, whenever $\epsilon^{i}\neq\mathds{1}$.
Suppose therefore that $\epsilon^{i}\neq\mathds{1}$. Then there exists
an element $a\in\F^{\times}_{q}$ such that $\epsilon^{i}(a)\neq1$.
Let 
\[
x=\left(\begin{smallmatrix}a\\
 & 1\\
 &  & \ddots\\
 &  &  & 1
\end{smallmatrix}\right)\in\GL_{m}(q).
\]
By general properties of the Steinberg character values (see \cite[9.3]{dignemichel}),
we have 
\[
\St_{m}(x)=|C_{\GL_{m}(q)}(x)|_{p}\neq0,
\]
so indeed, $(\epsilon^{i}\St_{m})(x)=\epsilon^{i}(a)\St_{m}(x)\neq\St_{m}(x)$.

The irreducibility of $\sigma$ now follows from \cite[Proposition~2.6]{Zelevinsky},
since the cuspidal support of $\St_{m}$ is $\mathds{1}_{\GL_{1}(\F_{q})}$
(with multiplicity $m$) and hence the cuspidal support of each $\epsilon^{i}\St_{m}$
is $\epsilon^{i}$ (with multiplicity $m$).
\end{proof}

Our second main result is the following.
\begin{thm}
\label{thm: Main thm 2}Every irreducible character of $\GL_{n}(q)$
that is trivial on the centre is contained in $\sigma^{2}$.
\end{thm}

We will give the proof of this theorem in the following section. Here
we only note that when $d=1$, Lemma~\ref{lem:The-maximal-order-of-lin-char}
says that, except for $\GL_{2}(2)$, the only linear character of
$\GL_{n}(q)$ that is trivial on $Z$ is the trivial character. Hence,
in this case, as $\sigma=\St_{n}$, Theorem~\ref{thm: Main thm 2} follows from Theorem~\ref{thm: Restriction to T} and Lemma~\ref{lem: St square contains all non-lin triv on Z} (together with a direct check for $\GL_{2}(2)$).

We also note that when $d=1$, it is well known that $\mathrm{PSL}_{n}(q)\cong\SL_{n}(q)\cong\PGL_{n}(q)$
is simple, except for $\mathrm{PGL}_{2}(2)$, so in this case Theorem~\ref{thm: Main thm 2}
also follows from, and coincides with, \cite[Theorem~1.2]{Heide-Saxl-Tiep-Zalesski}
(together with a direct check for $\mathrm{PGL}_{2}(2)$).

Identifying characters of $\GL_{n}(q)$ that are trivial on $Z$ with
characters of $\PGL_{n}(q)$, Theorem~\ref{thm: Main thm 2} says
that every irreducible character of $\PGL_{n}(q)$ is contained in
$\sigma^{2}$. However, $\sigma$ itself has central character $\epsilon^{m(d-1)d/2}$,
so need not be trivial on $Z$. In fact, as the following example
shows, there is in general no irreducible character of $\PGL_{2}(q)$
whose square contains every irreducible character.
\begin{example}
\label{exa:PGL2}Suppose that $n=2$ and that $q$ is odd (so that
$d=2$) and that $q\not\equiv1\pmod 4$. We show that there is no
irreducible character of $\PGL_{2}(q)$ whose square contains every
irreducible character. On the other hand, we also show that there
does exist an irreducible character of $G=\GL_{2}(q)$ whose square
contains every irreducible character of $G$ that is trivial on the
centre. By inspection of the character table of $\PGL_{2}(q)$, one
sees that the sum of the degrees of all the irreducible characters
is $q^{2}+1$, so if $\rho\in\Irr(\PGL_{2}(q))$ has square that contains
every irreducible character, then $\rho$ must have degree $q+1$
(and in particular, the square of the Steinberg character does not
contain every irreducible character of $\PGL_{2}(q)$). Write $B$
for the upper triangular subgroup and $T$ for the diagonal subgroup,
so that $R^{G}_{T}=\Ind^{G}_{B}\Inf^{B}_{T}(-)$. The irreducible
characters of $\PGL_{2}(q)$ of degree $q+1$, that is, the irreducible
characters of $G$ trivial on $Z$, are of the form $R^{G}_{T}(\alpha,\alpha^{-1})$,
where $\alpha\in\Irr(\F^{\times}_{q})$ is such that $\alpha\neq\alpha^{-1}$
and $(\alpha,\alpha^{-1})\in\Irr(T)$ is trivial on $Z$. 

Let $\epsilon\in\Irr(\F^{\times}_{q})$ be of order $2$, so that
$\epsilon\circ\det$ is a linear character of $G$ of order $2$ that
is trivial on $Z$. We show that either $\mathds{1}$ or $\epsilon\circ\det$
is not a contained in $R^{G}_{T}(\alpha,\alpha^{-1})^{2}$. Indeed,
\[
\langle\mathds{1},R^{G}_{T}(\alpha,\alpha^{-1})^{2}\rangle=\langle R^{G}_{T}(\alpha,\alpha^{-1}),\bar{R^{G}_{T}(\alpha,\alpha^{-1})}\rangle,
\]
which is $1$ if and only if $\text{\ensuremath{R^{G}_{T}}(\ensuremath{\alpha},\ensuremath{\alpha^{-1}})}$
is self-dual. Assume that $R^{G}_{T}(\alpha,\alpha^{-1})$ is self-dual.
Then 
\begin{multline*}
\langle\epsilon\circ\det,R^{G}_{T}(\alpha,\alpha^{-1})^{2}\rangle=\langle\epsilon R^{G}_{T}(\alpha,\alpha^{-1}),R^{G}_{T}(\alpha,\alpha^{-1})\rangle\\
=\sum_{g\in B\backslash G/B}\langle\epsilon\Inf^{B}_{T}(\alpha,\alpha^{-1})|_{B\cap\leftexp{g}{B}},\leftexp{g}{\Inf^{B}_{T}(\alpha,\alpha^{-1})}|_{B\cap\leftexp{g}{B}}\rangle\\
=\langle(\epsilon\circ\det|_{T})(\alpha,\alpha^{-1}),(\alpha,\alpha^{-1})\rangle+\langle(\epsilon\circ\det|_{T})(\alpha,\alpha^{-1}),(\alpha^{-1},\alpha)\rangle.
\end{multline*}
Here the first term is $0$ since $\epsilon\circ\det|_{T}\neq\mathds{1}_{T}$
and the second term is $0$ since if $\epsilon\circ\det|_{T}=(\alpha^{-2},\alpha^{2})$,
then $\epsilon(ab)=\alpha^{-2}(a)\alpha^{2}(b)$ for all $a,b\in\F^{\times}_{q}$,
so taking $a=b^{-1}$ implies that $\alpha^{4}=\mathds{1}$. But since
$q\not\equiv1\pmod 4$, there is no element of $\Irr(\F^{\times}_{q})$
of order $4$. Thus $\alpha^{2}=\mathds{1}$, contradicting $\alpha\neq\alpha^{-1}$.
Therefore $\epsilon\circ\det$ is not contained in $R^{G}_{T}(\alpha,\alpha^{-1})^{2}$.

On the other hand, now drop the condition $q\not\equiv1\pmod 4$ and
consider the character $\sigma=R^{G}_{T}(\mathds{1},\epsilon)$ of
$G$ (see Definition~\ref{def: Char sigma}; note that $\sigma$
is not trivial on $Z$). By direct computation of the character values
as in \cite[Theorem~3.1~(3)]{Aburto-Hageman-PantojaGL2} or by \cite[Theorem~1.1~(8)]{Gupta2025},
\[
\sigma^{2}=\Ind^{G}_{T}(\mathds{1},\epsilon^{2})+R^{G}_{T}(\epsilon,\epsilon)=\Ind^{G}_{T}\mathds{1}+\epsilon(\mathds{1}+\St_{2}).
\]
 It then follows from Theorem~\ref{thm: Restriction to T} that every
irreducible character of $\PGL_{2}(q)$ is either contained in $\Ind^{G}_{T}\mathds{1}$
or is linear (i.e., equal to $\epsilon\circ\det$), thus in either
case is contained in $\sigma^{2}$.
\end{example}

\begin{rem}
The preceding example shows in particular that $\PGL_{2}(3)$ does
not have the `tensor square property', even though this group has
the property that every irreducible character is contained in the
conjugation character (see \cite[Theorem~1.10]{Passman-conj_rep}
and note that $\PGL_{2}(3)\cong S_{4}$, or inspect the character
table to see that the sum of each row is non-zero). As shown in \cite{Heide-Saxl-Tiep-Zalesski},
these two properties are equivalent for all simple finite group of
Lie type, as the counterexamples that appear are exactly the same
for both properties.
\end{rem}

\section{Proof of Theorem~\ref{thm: Main thm 2}}

Let $P$ be a parabolic subgroup of $G$ (i.e., the $\F_{q}$-points
of a parabolic subgroup of the algebraic group $\GL_{n}$ over $\F_{q}$)
and $L$ a Levi subgroup of $P$. For $\lambda\in\Irr(L)$, let $\tilde{\lambda}=\Inf^{P}_{L}(\lambda)$
be its inflation to $P$ and $R^{G}_{L}(\lambda)=\Ind^{G}_{P}\tilde{\lambda}.$
By the well-known formula $\chi\Ind^{G}_{H}\rho=\Ind^{G}_{H}(\chi|_{H}\rho)$
and Mackey's induction-restriction formula, 

\begin{align*}
R^{G}_{L}(\lambda)^{2} & =\Ind^{G}_{P}(\tilde{\lambda}R^{G}_{L}(\lambda)|_{P})=\sum_{g\in P\backslash G/P}\Ind^{G}_{P}(\tilde{\lambda}\Ind^{P}_{P\cap\leftexp{g}{P}}\leftexp{g}{\tilde{\lambda}})\\
 & =\sum_{g\in P\backslash G/P}\Ind^{G}_{P\cap\leftexp{g}{P}}(\tilde{\lambda}|_{P\cap\leftexp{g}{P}}\leftexp{g}{\tilde{\lambda}}),
\end{align*}
where $\leftexp{g}{\tilde{\lambda}}$ denotes the restriction to $P\cap\leftexp{g}{P}$
of the representation of $\leftexp{g}{P}$ obtained from $\tilde{\lambda}$.
Thus, in particular, for each $g\in G$, 
\begin{equation}
R^{G}_{L}(\lambda)^{2}\supseteq\Ind^{G}_{P\cap\leftexp{g}{P}}(\tilde{\lambda}|_{P\cap\leftexp{g}{P}}\leftexp{g}{\tilde{\lambda}}).\label{eq:sigma-squared formula}
\end{equation}

From now on, we specialise to the situation considered in Definition~\ref{def: Char sigma},
that is, $P$ is the standard parabolic subgroup with $d$ blocks
of size $m$. Let $L=\GL_{m}(q)^{d}$ be the block diagonal subgroup
consisting of $d$ blocks of size $m$, let $\lambda=(\St_{m},\epsilon\St_{m},\dots,\epsilon^{d-1}\St_{m})\in\Irr(L)$
and put $\sigma=R^{G}_{L}(\lambda)$.
\begin{prop}
\label{prop:linear-general} Let $\chi\in\Irr(\F^{\times}_{q})$ be
such that $\chi\circ\det$ is trivial on $Z$. Then there exists a
$g\in N_{G}(L)$ such that 
\[
\lambda\leftexp{g}{\lambda}=(\chi\circ\det)|_{L}(\St^{2}_{m},\ldots,\St^{2}_{m}).
\]
Consequently, $\sigma^{2}$ contains every linear character of $G$
that is trivial on $Z$.
\end{prop}

\begin{proof}
We first note that there exists an $i\in\{0,\dots,d-1\}$ such that
$\chi=\epsilon^{i}$. Indeed, since $\chi\circ\det$ is trivial on
$Z$, we have $\chi^{n}=\mathds{1}$ and since also $\chi^{q-1}=\mathds{1}$,
it follows that $\ord(\chi)\mid\gcd(n,q-1)=d$. Thus $\chi$ lies
in the subgroup of $\Irr(\F^{\times}_{q})$ of order $d$  and by
assumption, $\epsilon$ is a generator of that subgroup. There exists
a permutation matrix $g\in N_{G}(L)$ such that 
\[
\leftexp{g}{\lambda}=(\epsilon^{d-j+i}\St_{m})^{d-1}_{j=0}=(\epsilon^{d+i}\St_{m},\dots,\epsilon^{i+1}\St_{m})=(\epsilon^{i}\St_{m},\dots,\epsilon^{i+1}\St_{m}).
\]
 Hence 
\[
\lambda\leftexp{g}{\lambda}=(\epsilon^{i}\St^{2}_{m},\dots,\epsilon^{i}\St^{2}_{m})=(\chi\circ\det)|_{L}(\St^{2}_{m},\ldots,\St^{2}_{m}).
\]

Let $U$ be the block-upper unipotent subgroup of $P$ such that $P=LU$
and $L\cap U=\{1\}$. Let $H=P\cap\leftexp{g}{P}$. Since $g\in N_{G}(L)$,
we have $H=L\bigl(U\cap\leftexp{g}{U}\bigr)$. As both $\tilde{\lambda}|_{H}$
and $\leftexp{g}{\tilde{\lambda}}$ are trivial on $U\cap\leftexp{g}{U}$,
\[
\tilde{\lambda}|_{H}\leftexp{g}{\tilde{\lambda}}=\Inf^{H}_{L}(\lambda\leftexp{g}{\lambda})=(\chi\circ\det)|_{H}\,\Inf^{H}_{L}(\St^{2}_{m},\ldots,\St^{2}_{m}).
\]
Consequently, by (\ref{eq:sigma-squared formula}), 
\begin{align*}
\sigma^{2} & \supseteq\Ind^{G}_{H}\big((\chi\circ\det)|_{H}\,\Inf^{H}_{L}(\St^{2}_{m},\ldots,\St^{2}_{m})\big)=\chi\Ind^{G}_{H}\Inf^{H}_{L}(\St^{2}_{m},\ldots,\St^{2}_{m})\\
 & \supseteq\chi\Ind^{G}_{H}\mathds{1}\supseteq\chi\circ\det.
\end{align*}
where we have used that $\St^{2}_{m}\supseteq\mathds{1}$, as $\langle\St^{2}_{m},\mathds{1}\rangle=\langle\St_{m},\St_{m}\rangle=1$.
This proves that $\sigma^{2}$ contains every linear character of
$G$ that is trivial on $Z$ for $G\neq\GL_{2}(2)$, as in this case
every linear character is of the form $\chi\circ\det$. For $G=\GL_{2}(2)$,
the assertion follows by a direct check, as $\sigma=\St_{2}$.
\end{proof}

\begin{lem}
\label{lem:Mackey-Levi-term} There exists $g\in N_{G}(L)$ such that
$P\cap\leftexp{g}{P}=L$ and 
\[
\lambda\leftexp{g}{\lambda}=(\varepsilon^{-1}\circ\det)|_{L}(\St^{2}_{m},\ldots,\St^{2}_{m}).
\]
Consequently, 
\[
\sigma^{2}\supseteq\varepsilon^{-1}\Ind^{G}_{L}(\St^{2}_{m},\ldots,\St^{2}_{m}).
\]
\end{lem}

\begin{proof}
Let $g\in N_{G}(L)$ be the block permutation matrix that reverses
the order of the $d$ blocks. Thus the permutation induced by $g$
on the set of blocks is $j\longmapsto d-1-j,\qquad j\in\{0,\ldots,d-1\}$.
Since $P$ is the subgroup of upper block-triangular matrices, $\leftexp{g}{P}$
is the subgroup of lower block-triangular matrices. Hence $P\cap\leftexp{g}{P}=L$.
Conjugation by $g$ reverses the components of $\lambda$. By a computation
similar to the one in the proof of Proposition~\ref{prop:linear-general},
\[
\lambda\leftexp{g}{\lambda}=(\epsilon^{d-1}\St^{2}_{m},\dots,\epsilon^{d-1}\St^{2}_{m})=(\epsilon^{-1}\circ\det)|_{L}(\St^{2}_{m},\ldots,\St^{2}_{m}).
\]

By (\ref{eq:sigma-squared formula}) and the fact that $P\cap\leftexp{g}{P}=L$,
\[
\sigma^{2}\supseteq\Ind^{G}_{L}(\lambda\leftexp{g}{\lambda})\supseteq\varepsilon^{-1}\Ind^{G}_{L}(\St^{2}_{m},\ldots,\St^{2}_{m}).
\]
\end{proof}

\begin{lem}
\label{lem:restriction-to-Levi} Let 
\[
M=\GL_{n_{1}}(q)\times\cdots\times\GL_{n_{r}}(q)
\]
be a Levi subgroup containing $T$. Write $T=T_{1}\times\cdots\times T_{r},$
where $T_{i}$ is the diagonal torus of $\GL_{n_{i}}(q)$. Suppose
that $\rho\in\Irr(G)$ contains $\mathds{1}_{T}$ on restriction to
$T$. Then there exists an irreducible constituent 
\[
\eta=(\eta_{1},\dots,\eta_{r})=\eta_{1}\boxtimes\cdots\boxtimes\eta_{r}
\]
of $\rho|_{M}$ such that for every $i$, $\eta_{i}$ contains $\mathds{1}_{T_{i}}$
on restriction to $T_{i}$. In particular, each $\eta_{i}$ is trivial
on the centre of $\GL_{n_{i}}(q)$. 
\end{lem}

\begin{proof}
Decompose $\rho|_{M}=\sum_{\eta\in\Irr(M)}m_{\eta}\eta$, where $m_{\eta}\geq0$.
Restricting further to $T$, we obtain $\rho|_{T}=\sum_{\eta\in\Irr(M)}m_{\eta}\,\eta|_{T}$
and thus 
\[
\langle\rho|_{T},\mathds{1}_{T}\rangle_{T}=\sum_{\eta\in\Irr(M)}m_{\eta}\langle\eta|_{T},\mathds{1}_{T}\rangle_{T}.
\]
By assumption, the left-hand side is positive, hence there exists
some irreducible constituent $\eta\in\Irr(M)$ such that $\langle\eta|_{T},\mathds{1}_{T}\rangle_{T}>0$.
Now every irreducible character of $M$ is of the form
\[
\eta=(\eta_{1},\dots,\eta_{r}),\qquad\eta_{i}\in\Irr(\GL_{n_{i}}(q)).
\]
As $T=T_{1}\times\cdots\times T_{r}$, we have $\eta|_{T}=(\eta_{1}|_{T_{1}},\dots,\eta_{r}|_{T_{r}})$
and therefore, 
\[
\langle\eta|_{T},\mathds{1}_{T}\rangle_{T}=\prod^{r}_{i=1}\langle\eta_{i}|_{T_{i}},\mathds{1}_{T_{i}}\rangle_{T_{i}}.
\]
Since this product is positive, each factor must be positive, that
is, $\langle\eta_{i}|_{T_{i}},\mathds{1}_{T_{i}}\rangle_{T_{i}}>0$
for all $i$.
\end{proof}

We can now finish the proof Theorem~\ref{thm: Main thm 2}.
\begin{proof}
Let $\rho\in\Irr(G)$ be trivial on $Z$. If $\rho$ is linear, the
result follows from Proposition~\ref{prop:linear-general}, so assume
that $\rho$ is nonlinear. 

Let $T=T_{1}\times\cdots\times T_{d}$ be the diagonal torus of $G$,
where $T_{i}$ is the diagonal torus of the $i$th factor $\GL_{m}(q)$
of $L=\GL_{m}(q)^{d}$. By Theorem~\ref{thm: Restriction to T},
$\rho$ contains $\mathds{1}_{T}$ on restriction to $T$, so by Lemma~\ref{lem:restriction-to-Levi},
$\rho|_{L}$ has an irreducible constituent $\eta=(\eta_{1},\dots,\eta_{d})$
such that $\left\langle \eta_{i}|_{T_{i}},\mathds{1}_{T_{i}}\right\rangle _{T_{i}}>0$
for every $i\in\{1,\ldots,d\}$. Thus by Lemma~\ref{lem: St square contains all non-lin triv on Z},
$\eta_{i}\subseteq\St^{2}_{m}$ for every $i$.
It follows that $\eta\subseteq(\St^{2}_{m},\ldots,\St^{2}_{m})$
and since $\eta$ is a constituent of $\rho|_{L}$, Frobenius reciprocity
implies that 
\[
\rho\subseteq\Ind^{G}_{L}\eta\subseteq\Ind^{G}_{L}(\St^{2}_{m},\ldots,\St^{2}_{m}).
\]

Since $\rho$ was an arbitrary nonlinear irreducible character of
$G$ trivial on $Z$ and $\varepsilon\circ\det$ is trivial on $Z$,
$\epsilon\rho$ is again nonlinear, irreducible and trivial on $Z$,
so the above argument also implies that 
\[
\epsilon\rho\subseteq\Ind^{G}_{L}(\St^{2}_{m},\ldots,\St^{2}_{m}).
\]
After twisting by $\epsilon^{-1}\circ\det$ and applying Lemma~\ref{lem:Mackey-Levi-term},
we thus obtain 
\[
\rho\subseteq\varepsilon^{-1}\Ind^{G}_{L}(\St^{2}_{m},\ldots,\St^{2}_{m})\subseteq\sigma^{2}.
\]
\end{proof}

\begin{cor}
\label{cor: central char extends then square}Suppose that $\chi\in\Irr(Z)$
extends to a character $\hat{\chi}$ of $G$. Then $(\hat{\chi}\sigma)^{2}$
contains all the irreducible characters of $G$ with central character
$\chi^{2}$.
\end{cor}

\begin{proof}
Let $\rho\in\Irr(G)$ have central character $\chi^{2}$. Then $\hat{\chi}^{-2}\rho$
is trivial on $Z$, so by Theorem~\ref{thm: Main thm 2}, $\hat{\chi}^{-2}\rho\subseteq\sigma^{2}$
and hence $\rho\subseteq(\hat{\chi}\sigma)^{2}$.
\end{proof}

Note that if $\rho\in\Irr(G)$ has central character $\chi$, then
every irreducible constituent of $\rho^{2}$ has central character
$\chi^{2}$. In view of the preceding results, it is natural to ask
whether, for any square character $\zeta\in\Irr(Z)$, there exists
a $\rho\in\Irr(G)$ whose square contains all the irreducible characters
of $G$ with central character $\zeta$. Corollary~\ref{cor: central char extends then square}
gives a positive answer when a square root of $\zeta$ extends to
$G$. The following example shows that the answer can otherwise be
negative.
\begin{example}
We show that $G=\GL_{2}(5)$ does not have any irreducible character
whose square contains all the irreducible characters of $G$ with
central character $\zeta$ of order $2$. We identify the centre $Z$
of $G$ with $\F^{\times}_{q}$ and identify the corresponding characters.
Let $\chi$ be an irreducible character of $Z$ of order $4$, so
that $\Irr(Z)=\{\mathds{1},\zeta,\chi,\bar{\chi}\}$. Note that $\chi^{2}=\bar{\chi}^{2}=\zeta$
and $\zeta\chi=\bar{\chi}$.

We first find the irreducible characters of $G$ with central character
$\zeta$. Note that $\zeta$ extends to $\chi\circ\det\in\Irr(G)$,
since $\chi\circ\det|_{Z}=\bar{\chi}\circ\det|_{Z}=\chi^{2}=\zeta$.
Thus, $\rho\in\Irr(G)$ has central character $\zeta$ if and only
if $\chi\rho$ is trivial on $Z$. Thus the characters of $G$ with
central character $\zeta$ are precisely $C=\{\chi\tilde{\rho}\mid\rho\in\Irr(G/Z)\}$,
where $\tilde{\rho}$ is the inflation and we note, as is well-known,
that $G/Z\cong S_{5}$.

Next, assume that there exists a $\sigma\in\Irr(G)$ such that $\sigma^{2}$
contains all the characters $\chi\tilde{\rho}\in C$. Then, since
the degrees of the characters in $C$ are $1,1,4,4,5,5,6$, we must
have $\sigma^{2}(1)\geq26$, hence, as the degrees of irreducible
characters of $G$ are $1,4,5$ and $6$, we must have $\sigma(1)\geq6$.
Moreover, since every irreducible constituent of $\sigma^{2}$ must
have central character $\zeta$, it follows that $\sigma$ has central
character $\chi$ or $\bar{\chi}$. We thus proceed to find the irreducible
characters of $G$ of degree $6$ and central character $\chi$ or
$\bar{\chi}$.

Let $\rho\in\Irr(G)$ be of degree $6$. Then $\rho=R^{G}_{T}(\alpha,\beta)$,
for some $\alpha,\beta\in\Irr(\F^{\times}_{q})$ with $\alpha\neq\beta$,
and $\rho$ has central character $\alpha\beta$. Assuming that $\alpha\beta=\chi$,
so that $\beta=\bar{\alpha}\chi$, the possibilities for $(\alpha,\beta)$
are $(\mathds{1},\chi),(\zeta,\bar{\chi}),(\chi,\mathds{1}),(\bar{\chi},\zeta)$.
This gives two distinct irreducible characters
\[
\rho_{1}=R^{G}_{T}(\mathds{1},\chi),\qquad\rho_{2}=R^{G}_{T}(\zeta,\bar{\chi}).
\]
Similarly, the possible irreducible characters $R^{G}_{T}(\alpha,\beta)$
with central character $\bar{\chi}$ are
\[
\rho_{3}=R^{G}_{T}(\mathds{1},\bar{\chi})=\bar{\rho}_{1},\qquad\rho_{4}=R^{G}_{T}(\zeta,\chi)=\bar{\rho}_{2}.
\]
Hence $\rho\in\{\rho_{1},\rho_{2},\rho_{3},\rho_{4}\}$.

Finally, we compute inner products of the squares of the $\rho_{i}$
with linear characters to derive a contradiction. Let $\phi\in\Irr(G/Z)$
denote the nontrivial linear character. Then $\tilde{\phi}=\zeta\circ\det$,
so $\chi\tilde{\phi}=\chi\zeta\circ\det=\bar{\chi}\circ\det$. Thus
\begin{align*}
\langle\rho^{2}_{1},\chi\tilde{\phi}\rangle & =\langle\rho_{1},\bar{\chi}\bar{\rho}_{1}\rangle=\langle\rho_{1},\bar{\chi}\rho_{3}\rangle=\langle\rho_{1},R^{G}_{T}(\bar{\chi}\circ\det)(\mathds{1},\bar{\chi})\rangle\\
 & =\langle\rho_{1},R^{G}_{T}(\bar{\chi},\zeta)\rangle=\langle\rho_{1},\rho_{2}\rangle=0
\end{align*}
and
\begin{align*}
\langle\rho^{2}_{2},\chi\tilde{\phi}\rangle & =\langle\rho_{2},\bar{\chi}\bar{\rho}_{2}\rangle=\langle\rho_{2},\bar{\chi}\rho_{4}\rangle=\langle\rho_{2},R^{G}_{T}(\bar{\chi}\circ\det)(\zeta,\chi)\rangle\\
 & =\langle\rho_{2},R^{G}_{T}(\chi,\mathds{1})\rangle=\langle\rho_{2},\rho_{1}\rangle=0.
\end{align*}
Moreover, from these computations we also deduce
\begin{align*}
\langle\rho^{2}_{3},\chi\circ\det\rangle & =\langle\rho_{3},\chi\bar{\rho}_{3}\rangle=\langle\rho_{3},\chi\rho_{1}\rangle=\langle\bar{\chi}\rho_{3},\rho_{1}\rangle=0
\end{align*}
and
\begin{align*}
\langle\rho^{2}_{4},\chi\circ\det\rangle & =\langle\rho_{4},\chi\bar{\rho}_{4}\rangle=\langle\rho_{4},\chi\rho_{2}\rangle=\langle\bar{\chi}\rho_{4},\rho_{2}\rangle=0.
\end{align*}
Therefore every candidate square omits at least one of the two required
linear characters, which is a contradiction.
\end{example}

\bibliographystyle{alex}
\bibliography{alex}

\end{document}